\documentclass[12pt,leqno]{article}
\usepackage{amsmath}
\usepackage{amssymb}
\usepackage{theorem}
\usepackage[scanall]{psfrag}
\usepackage{epsf}
\usepackage{graphicx}

\newcommand\qed{\hfill\vrule height6pt width6pt depth0pt}

\newcommand{\supp}{{\rm supp}}
\newcommand{\diam}{{\rm diam}}
\newcommand{\vphi}{{\varphi}}
\newcommand{\rf}[1]{{(\ref{#1})}}
\newcommand{\wt}[1]{{\widetilde{#1}}}

\newcommand{\res}{\hbox{ {\vrule height .22cm}{\leaders\hrule\hskip.2cm} } }
\newcommand{\RR}{\mathbb{R}}
\newcommand{\NN}{\mathbb{N}}
\newcommand{\ve}{{\varepsilon}}
\newcommand{\dist}{\mathrm{dist}}

\newcommand{\support}{\mathrm{supp}}
\renewcommand{\H}{\mathcal{H}}

\newcommand{\U}{\mathcal{U}}

\theoremstyle{plain}
\newtheorem{thm}{Theorem}[section]
\newtheorem{lem}[thm]{Lemma}
\newtheorem{cor}[thm]{Corollary}
\newtheorem{prop}[thm]{Proposition}
{\theorembodyfont{\rmfamily}
\newtheorem{defn}[thm]{Definition}}
{\theorembodyfont{\rmfamily}
\newtheorem{rem}[thm]{Remark}}
\newtheorem{case}{Case}

\newenvironment{proof}{\medskip\noindent{\bf Proof: }\rmfamily}{\medskip}

\numberwithin{equation}{section}

\begin{document}

\title{On the smoothness of H\"older-doubling measures}
\author{D.\ Preiss,\and  X.\ Tolsa,
\and T.\ Toro
\thanks{X.\ T.\ was partially supported by grants MTM2004-00519 (Spain) and 2005-SGR-00744 (Generalitat
de Catalunya). T.\ T.\ was partially supported by the NSF
under Grant DMS-0244834.}
}
\date{}
\maketitle

\section{Introduction}

In this paper we consider the question of whether the doubling character of a
measure supported on a subset of $\RR^m$ determines the regularity of
its support (in a classical sense). This problem was studied in
\cite{DKT} for codimension 1 sets under the assumption that the support
be flat. Here we study the higher codimension case and remove the flatness
hypothesis.

In order to give precise statements we need to introduce some definitions.
Fix integer dimensions $0<n<m$ and a closed set $\Sigma\subset\RR^m$. For
$x\in\Sigma$ and $r>0$, set
\begin{equation}\label{eqn:1.2}
\theta_{\Sigma}(x,r) = \frac{1}{r}\inf\{D[\Sigma\cap B(x,r), L\cap B(x,r)];
L\hbox{ is an affine }n\hbox{-plane through }x\},
\end{equation}
where $B(x,r)$ denotes the open ball of center $x$ and radius $r$ in $\RR^m$,
and where
\begin{equation}\label{eqn:1.3}
D[E,F]  =  \sup\{\dist(y,F); y\in E\}
+ \sup\{\dist(y,E); y\in F\}
\end{equation}
denotes the usual Hausdorff distance between (nonempty) sets. If there is no
ambiguity over the set we are considering we write $\theta(x,r)$  rather than
$\theta_{\Sigma}(x,r)$.

\begin{defn}\label{defn:1.4}
Let $\delta>0$ be given. We say that the closed set $\Sigma\subset\RR^m$ is
$\delta$-\emph{Reifenberg flat} of dimension $n$ if for all compact sets
$K\subset\Sigma$ there is a radius $r_K>0$ such that
\begin{equation}\label{eqn:1.5}
\theta(x,r)\le\delta\qquad\hbox{for all}\qquad x\in K\quad\hbox{and}\quad
0<r\le r_K.
\end{equation}
\end{defn}

Note that it does not make sense to take $\delta$ large (like $\delta\ge
2$), because $\theta(x,r)\le 2$ anyway.

\begin{defn}\label{defn:1.6}
We say that the closed set $\Sigma\subset\RR^m$ is
\emph{Reifenberg flat with vanishing constant}
(of dimension $n$) if for every compact subset
$K$ of $\Sigma$,
\begin{equation}\label{eqn:1.7}
\lim_{r\to 0^+} \theta_K(r)=0,
\end{equation}
where
\begin{equation}\label{eqn:1.?}
\theta_K(r) = \sup_{x\in K} \theta(x,r).
\end{equation}
\end{defn}

Unless otherwise specified,
``measure'' here will mean ``positive Radon measure'', i.e. ``Borel measure
which is finite on compact sets.'' Let $\mu$ be a measure on $\RR^m$, set
\begin{equation}\label{eqn:1.8}
\supp (\mu) = \{x\in\RR^m; \mu(B(x,r))>0\hbox{ for all }r>0\}.
\end{equation}

For a measure $\mu$ on $\RR^m$, with support $\Sigma=
\supp(\mu)$ we define for $x\in\Sigma$, $r>0$ and $t\in (0,1]$
the quantity
\begin{equation}\label{eqn:1.11}
R_t(x,r) = \frac{\mu(B(x,tr))}{\mu(B(x,r))}-t^n,
\end{equation}
which encodes the doubling properties of $\mu$.

\begin{defn}\label{defn:1.6A}
A measure $\mu$ supported on $\Sigma$ is said to be asymptotically
optimally doubling if
for each compact set
$K\subset \Sigma$, $x\in K$, and $t\in [\frac{1}{2},1]$
\begin{equation}\label{eqn:1.10}
\lim_{r\to 0^+}\sup_{x\in K}|R_t(x,r)|=0.
\end{equation}
\end{defn}

The results in this paper can be summarized as follows: first
under the appropriate conditions on $\theta(x,r)$
(see (\ref{eqn:1.2})) the asymptotic behavior of $R_t(x,r)$ as $r$
tends to 0 fully determines the regularity of $\Sigma$.
Second for asymptotically doubling measures which are Ahlfors regular
flatness is an open condition.

We mention the local versions of
some of the previous results along these lines.

\begin{thm}[\cite{KT}, \cite{DKT}]\label{thm:1.16}
Let $\mu$ be an asymptotically doubling  measure supported on
$\Sigma\subset \RR^{m}$.
If $n=1,2$, $\Sigma$ is Reifenberg flat with
vanishing constant.
If $n\ge 3$, there exists a constant $\delta(n,m)$ depending
only on $n$ and $m$ such that if $x_0\in\Sigma$ and $\Sigma\cap B(x_0, 2R_0)$
is $\delta(n,m)$-Reifenberg flat, then $\Sigma\cap B(x_0, R_0)$
is Reifenberg flat with vanishing constant.
\end{thm}

The converse is also true.

\begin{thm}[\cite{DKT}]\label{thm:1.12}
If $\Sigma$ is a Reifenberg flat set with vanishing constant
there exists a
measure $\mu$ supported on $\Sigma$ which satisfies (\ref{eqn:1.10}).
\end{thm}

Precise asymptotic estimates on the quantity $R_t(x,r)$
yield stronger results about the regularity of $\Sigma$.

\begin{thm}[\cite{DKT}]\label{thm:1.17}
For each constant $\alpha>0$ we can find $\beta=\beta(\alpha)>0$ with the
following property. Let $\mu$ be a measure in $\RR^{n+1}$,
set $\Sigma=\support(\mu)$, and suppose that for each compact set
$K\subset\Sigma$, there is a constant $C_K$ such that
\begin{equation}\label{eqn:1.1}
\left|\frac{\mu(B(x,tr))}{\mu(B(x,r))} - t^n\right|\le
C_Kr^\alpha \ \ \ \hbox{for}\
r\in (0,1],\ t\in[\frac{1}{2},1]\ \hbox{ and}\  x\in K.
\end{equation}
If $n=1,2$, $\Sigma$ is
a $C^{1,\beta}$ submanifold of dimension $n$ in $\RR^{n+1}$.
If $n\ge 3$, for $x_0\in\Sigma$ if $\Sigma\cap B(x_0, 2R_0)$
is $\frac{1}{4\sqrt{2}}$-Reifenberg flat, then $\Sigma\cap B(x_0, R_0)$
is a $C^{1,\beta}$ submanifold of dimension $n$ in $\RR^{n+1}$.
\end{thm}

For $n\geq 3$, the preceding theorem fails if one removes the flatness assumption. Indeed, Kowalski and Preiss \cite{KoP}
discovered that the $3$-dimensional Hausdorff $\H^3$ measure on the cone
$X=\{x\in \RR^4:x_4^2 = x_1^2+x_2^2+ x_3^2\}$
satisfies $\H^3(B(x,r)\cap X)=Cr^3$ for all $x\in X$ and all $r>0$. Clearly,
\eqref{eqn:1.1} holds in this case and $X$ is non smooth at the origin.

In this paper we extend Theorem \ref{thm:1.17} to general codimensions in
$\RR^m$, and moreover we prove that, when $n\geq3$, if one does not assume $\Sigma$ to be Reifenberg flat,
one still has that $\Sigma$ is smooth out of a small closed set (like in the case of the cone $X$). The precise
statement is the following.

\begin{thm}\label{thm:1.17A}
For each constant $\alpha>0$ we can find $\beta=\beta(\alpha)>0$ with the
following property. Let $\mu$ be a measure in $\RR^{m}$ supported on
$\Sigma$, and suppose that for each compact set
$K\subset\Sigma$, there is a constant $C_K$ such that
\begin{equation}\label{eqn:1.1A}
\left|\frac{\mu(B(x,tr))}{\mu(B(x,r))} - t^n\right|\le
C_Kr^\alpha \ \ \ \hbox{for}\
r\in (0,1],\ t\in[\frac{1}{2},1]\ \hbox{ and}\  x\in K.
\end{equation}
If $n=1,2$, $\Sigma$ is
a $C^{1,\beta}$ submanifold of dimension $n$ in $\RR^{m}$.
If $n\ge 3$, $\Sigma$ is a $C^{1,\beta}$ submanifold of dimension $n$
in $\RR^{m}$ away from a closed set ${\cal S}$ such that
${\cal H}^n({\cal S})=0$.
\end{thm}

\noindent We would like to point out that condition (\ref{eqn:1.1A}) implies an
apparently stronger condition, namely that for each compact set
$K\subset\Sigma$, there is a constant $C_K$ depending on
$K$, $n$ and $\alpha$ such that
\begin{equation}\label{eqn:1.100}
\left|\frac{\mu(B(x,tr))}{\mu(B(x,r))} - t^n\right|\le
C_Kr^\alpha \ \ \ \hbox{for}\
r\in (0,1],\ t\in (0,1]\ \hbox{ and}\  x\in K.
\end{equation}
In fact assume that (\ref{eqn:1.1A}) holds
and let $\tau\in (0,1/2) $. There exits
$j\in\NN $, $j\ge 2$ so that $1/2^{j}\le \tau <1/2^{j-1}$ thus
$\tau^{\frac{1}{j}}=t\in [1/2,1/\sqrt{2})$. For $x\in K$, and
$r\in (0,1]$, (\ref{eqn:1.1A})
yields
\begin{eqnarray}\label{eqn:1.101}
t^{n(j-1)}\left|\mu(B(x,tr))-t^{n}\mu(B(x,r))\right|&\le & C_Kr^\alpha
t^{n(j-1)} \mu(B(x,r))\\
t^{n(j-2)}\left|\mu(B(x,t^{2}r))-t^{n}\mu(B(x,tr))\right|&\le & C_Kr^\alpha
t^{n(j-2)}
\mu(B(x,tr))\nonumber\\
...\ \ \ &\le & \ \ \ ...\nonumber\\
\left|\mu(B(x,t^{j}r))-t^{n}\mu(B(x,t^{j-1}r))\right|&\le &
C_Kr^\alpha \mu(B(x,t^{j-1}r)).\nonumber
\end{eqnarray}
Adding the above inequalities, we obtain that
\begin{equation}\label{eqn:1.101b}
\left|\mu(B(x,\tau r))-\tau^{n}\mu(B(x,r))\right|\le  C_Kr^\alpha
\mu(B(x,r))\sum_{i=0}^{j-1}
\left(\frac{1}{(\sqrt{2})^n}\right)^{i}
\end{equation}
which implies that for $x\in (0,R)$, $x\in K$ and $\tau\in(0, {1/2})$
\begin{equation}\label{eqn:1.102}
\left|\frac{\mu(B(x,\tau r))}{\mu(B(x,r))}-\tau^{n}\right| \le
CC_Kr^\alpha.
\end{equation}
The constant $C$ depends only on the dimension $n$.

A first step in the proof of Theorem \ref{thm:1.17A} is to prove that if
$\mu$ satisfies (\ref{eqn:1.1A}), then the restriction $\mu_0$ of $\H^n$ to
$\Sigma$ is locally finite, and $d\mu(x)=D(x)d\mu_0(x)$ for some positive
density $D(x)$ such that $\log D(x)$ is (locally) H\"older with exponent
$\frac{\alpha}{\alpha+1}$. Moreover, $\mu_0$ satisfies the stronger
requirement that for all compact sets $K\subset\Sigma$ there is a constant
$C_K$ such that
\begin{equation}\label{eqn:1.19}
\left |\frac{\mu_0(B(x,r))}{\omega_nr^n}-1\right |\le C_K r^{\frac{\alpha}{\alpha+1}}
\hbox{ for
all }x\in K\hbox{ and }0<r\le 1.
\end{equation}

We then deduce the conclusion of Theorem \ref{thm:1.17A} from
(\ref{eqn:1.19}), by a method inspired by \cite{KoP}, \cite{DKT} and \cite{Pr}.
We first show that if $\mu$ is a Radon measure supported on
$\Sigma\subset \RR^m$, the local behavior of the quantity
$\frac{\mu(B(x,r))}{\omega_nr^n}$ for $x\in\Sigma$ and $r\in(0,1]$ determines
the regularity of $\Sigma$ near flat points. The we prove that the set of flat
points is open and its complement has ${\cal H}^n$ measure 0 (see the two
theorems below). Here $\omega_n$ denotes the Lebesgue measure of
the unit ball in $\RR^n$.

\begin{thm}\label{thm:7.1}
For each $\alpha>0$ there exists $\beta=\beta(\alpha)>0$ with the following
property. Suppose $\Sigma=$supp$(\mu)\subset\RR^{m}$ for some positive
Radon measure $\mu$, and that for each compact set $K\subset\Sigma$ there is
a constant $C_K$ such that
\begin{equation}\label{eqn:7.2}
\left|\frac{\mu(B(x,r))}{\omega_nr^n}-1\right|\le C_Kr^\alpha
\end{equation}
for $x\in K$ and $0<r<1$.
If $n=1,2$, $\Sigma$ is
a $C^{1,\beta}$ submanifold of dimension $n$ in $\RR^{m}$.
If $n\ge 3$, there exists a constant $\delta(n,m)$ depending
only on $n$ and $m$ such that if $x_0\in\Sigma$ and $\Sigma\cap B(x_0, 2R_0)$
is $\delta(n,m)$-Reifenberg flat, then $\Sigma\cap B(x_0, R_0)$
is a $C^{1,\beta}$ submanifold of dimension $n$ in $\RR^{m}$.
\end{thm}

\begin{thm}\label{thm:1.8A}
For each $\alpha>0$ there exists $\beta=\beta(\alpha)>0$ with the following
property. Suppose $\Sigma=$supp$(\mu)\subset\RR^{m}$ for some positive
Radon measure $\mu$, and that for each compact set $K\subset\Sigma$ there is
a constant $C_K$ such that
\begin{equation}\label{eqn:7.2AA}
\left|\frac{\mu(B(x,r))}{\omega_nr^n}-1\right|\le C_Kr^\alpha
\end{equation}
for $x\in K$ and $0<r<1$.
If $n=1,2$, $\Sigma$ is
a $C^{1,\beta}$ submanifold of dimension $n$ in $\RR^{m}$.
If $n\ge 3$, $\Sigma$ is a $C^{1,\beta}$ submanifold of dimension $n$
in $\RR^{m}$ away from a closed set ${\cal S}$ such that
${\cal H}^n({\cal S})=0$.
\end{thm}

\section{Preliminaries}

In this section we state several results which will be used throughout the
paper. The codimension one versions appear in \cite{DKT}. The reader would
realize that the proofs given in there do not depend on the codimension.
Thus we do not include proofs.

\begin{prop}\label{prop:6.2}
Let $\alpha>0$ be given. Let $\mu$ be a measure supported on
$\Sigma\subset\RR^m$ and suppose that for all compact sets $K\subset
\Sigma$, there is a constant $C_K$ such that
\begin{equation}\label{eqn:6.3}
|R_t(x,r)|\le C_Kr^\alpha\hbox{ for }x\in K\hbox{ and }t, r\in(0,1].
\end{equation}
Then the density
\begin{equation}\label{eqn:6.4}
D(x)=\lim_{r\to 0^+} \frac{\mu(B(x,r))}{\omega_nr^n}
\end{equation}
(where $\omega_n$ denotes the $n$-dimensional Hausdorff measure of
the unit ball
in $\RR^n$) exists for all $x\in\Sigma$, and
\begin{equation}\label{eqn:6.5}
0<D(x)<+\infty\ \hbox{ for }\ x\in \Sigma.
\end{equation}
Moreover, $\log D(x)$ is locally H\"older; i.e., for all compact sets
$K\subset\Sigma$, we can find $C'_K$ such that
\begin{equation}\label{eqn:6.6}
|\log D(x)-\log D(y)|\le C'_K|x-y|^{\frac{\alpha}{1+\alpha}}\ \hbox{ for }\  x,
y\in K.
\end{equation}
Finally, denote by $\mu_0$ the restriction of $\H^n$ to $\Sigma$, i.e.,
$\mu_0=\H^n \res\Sigma$. Then $\mu_0$ is finite on compact sets,
\begin{equation}\label{eqn:6.7}
d\mu(x)=D(x)d\mu_0(x),
\end{equation}
and for each compact set $K\subset\Sigma$ there is a constant $C''_K$ such
that
\begin{equation}\label{eqn:6.8}
\left|\frac{\mu_0(B(x,r))}{\omega_nr^n}-1\right|\le C''_K
r^{\frac{\alpha}{1+\alpha}}\hbox{ for }x\in K \hbox{ and }0<r\le 1.
\end{equation}
\end{prop}

\begin{rem}\label{rem:6.9}
When $C_K$ in (\ref{eqn:6.3}) is  large, (\ref{eqn:6.3}) only gives some
information on the doubling properties of $\mu$ at small scales (i.e., when
$r^\alpha<C^{-1}_K$). Thus, even though we did not say explicitly that
$R_t(x,r)$ is only controlled for $r$ small enough, this is implicit in
(\ref{eqn:6.3}).

Since (\ref{eqn:6.6}) and (\ref{eqn:6.8}) contain some amount of large-scale
information, we might be forced in some cases to take huge values of $C'_K$
and $C''_K$, that depend on the large-scale behavior of $\mu$ (and not
only on the $C_K$). This problem can easily be fixed by restricting the
domain of validity of (\ref{eqn:6.6}) to $|x-y|\le r_0$, where $r_0$ depends
on $C_K$, and similarly restricting (\ref{eqn:6.8}) to radii $0<r<r_0$. Then
we can get constants $C'_K$ and $C''_K$ that depend only on $C_K$. We could
also fix the problem by requiring that $\mu$ be doubling.
\end{rem}

\noindent
Let $\mu$ be an $n$-Ahlfors regular measure supported on $\Sigma\subset\RR^m$,
i.e suppose that
for each compact set $K\subset\Sigma$ there is
a constant $C_K>1$ such that
\begin{equation}\label{eqn:7.2A}
C_K^{-1}<\frac{\mu(B(x,r))}{\omega_nr^n}<C_K
\end{equation}
for $x\in K$ and $0<r<1$.
We follow \cite{KoP} and introduce some moments for Ahlfors regular measures.
Fix a compact set $K$ and for $x_1\in K$, define the vector $b=b_{x_1,r}$ by
\begin{equation}\label{eqn:7.3}
b=\frac{n+2}{2\omega_nr^{n+2}}\int_{B(x_1,r)}(r^2-|y-x_1|^2)(y-x_1)d\mu(y).
\end{equation}
Also define the quadratic form $Q=Q_{x_1,r}$ on $\RR^{m}$ by
\begin{equation}\label{eqn:7.4}
Q(x)=\frac{n+2}{\omega_nr^{n+2}}\int_{B(x_1,r)}\langle x, y-x_1\rangle^2d\mu(y)
\end{equation}
for $x\in\RR^{m}$.
In all our estimates we use the fact that
\begin{equation}\label{eqn:7.5}
|\mu(B(x,t))-\omega_nt^n|\le C_Kt^{n+\alpha}\hbox{ for }x\in\Sigma\cap\overline
B(x_1,1)\hbox{ and }0<t<1,
\end{equation}
which we get by applying (\ref{eqn:7.2}) with $K^\ast=\{x\in\Sigma;
\dist(x,K)\le 1\}$.

Roughly speaking the following proposition shows that if the density
ratio of $\mu$,
$\frac{\mu(B(x,r))}{\omega_nr^n}$ approaches 1 as $r$ tends to 0 in a
H\"older fashion then the points in the support of $\mu$ almost satisfy
a quadratic equation.

\begin{prop}\label{prop:7.6}
Let $\mu$ be a measure supported on $\Sigma\subset\RR^m$ such that
for each compact set $K\subset\Sigma$ there is
a constant $C_K$ such that
\begin{equation}\label{eqn:7.2B}
\left|\frac{\mu(B(x,r))}{\omega_nr^n}-1\right|\le C_Kr^\alpha
\end{equation}
for $x\in K$ and $0<r<1$. For $x_1\in K$ and $0<r<1$, let
\begin{equation}\label{eqn:7.7}
Tr(Q)=\frac{n+2}{\omega_nr^{n+2}}\int_{B(x_1,r)}|y-x_1|^2d\mu(y)
\end{equation}
denote the trace of $Q$. Then
\begin{equation}\label{eqn:7.8}
|Tr(Q)-n|\le CC_Kr^\alpha.
\end{equation}
Also, if $0<r<\frac{1}{2}$, for $x\in\Sigma\cap
B\left(x_1,\frac{r}{2}\right)$,
\begin{equation}
\label{eqn:7.9}
\left|2\langle b, x-x_1\rangle+Q(x-x_1)-|x-x_1|^2\right|\le C
\frac{|x-x_1|^3}{r}+CC_Kr^{2+\alpha}.
\end{equation}
\end{prop}

For a measure $\mu$ supported on $\Sigma$ and satisfying (\ref{eqn:7.2B})
we introduce the quantity that allows us to measure the local flatness
of $\Sigma$
and prove its regularity.
Let $K\subset\Sigma$ be a compact set and let $x_1\in K$, for small radii
$\rho$ consider
\begin{equation}\label{eqn:8.1}
\beta(x_1,\rho) = \inf_P\{\frac{1}{\rho}\sup\{\dist(y,P); y\in\Sigma\cap
B(x_1,\rho)\}\}
\end{equation}
Here the infimum is taken over all affine $n$-planes $P$ through $x_1$.
In particular by (\ref{eqn:1.2})
\[
\beta(x_1,\rho)\le \theta(x_1,\rho).
\]
Note that (\ref{eqn:7.2}) implies that $\mu$
satisfies the hypothesis of
Theorem \ref{thm:1.16} which ensures that if $\Sigma\cap B(x_0, 2R_0)$ is
Reifenberg flat for some $x_0\in \Sigma$ then
\begin{equation}\label{eqn:8.2}
\Sigma\cap B(x_0, R_0)\hbox{ is Reifenberg flat with vanishing constant.}
\end{equation}
Hence for $x_1\in \Sigma\cap B(x_0, R_0)$, $\beta(x_1,\rho)$
converges to $0$ as $\rho\rightarrow 0$ uniformly on
compact sets. The key step in the proof of Theorem
\ref{thm:1.17A} is to show that if $\mu$ satisfies (\ref{eqn:1.1A}) then
there exists $\gamma>0$ such that for $\rho$ small
$\beta(x_1,\rho)<C_K \rho^{\gamma}$. This is also the main idea behind
the proof of Theorem \ref{thm:1.17}. Its implementation in the codimension
1 case is significantly simpler. Once the asymptotic behavior of $\beta$ has
been established we simply apply the following theorem which appears in
Section 9 in \cite{DKT}

\begin{prop}\label{prop:9.1}
Let $0<\beta\le 1$ be given. Suppose $\Sigma\cap B(x_0, 2R_0)$ is
a Reifenberg flat set with
vanishing constant of dimension $n$ in $\RR^m$ and that, for each
compact set $K\subset\Sigma$, there is a constant $C_K$ such that
\begin{equation}\label{eqn:9.2}
\beta(x,r)\le C_Kr^\beta\hbox{ for }x\in K\hbox{ and }r\le 1.
\end{equation}
Then $\Sigma\cap B(x_0, R_0)$ is a $C^{1,\beta}$
submanifold of dimension $n$ of $\RR^m$.
\end{prop}

\section{Control on the flatness of $\Sigma$}

Let $\mu$ be a measure satisfying the hypothesis of Proposition
\ref{prop:7.6}. Assume
$\mu$ is supported on $\Sigma\subset\RR^m$,
and let $K\subset \Sigma$ be a fixed
compact set. Theorem \ref{thm:1.16} ensures that
for each small $\delta>0$, we can find $r_0\in(0,10^{-2}R_0)$ depending on $K$
such that
\begin{equation}\label{eqn:8.3}
\theta(x,r)\le\delta\hbox{ when }x\in\Sigma\cap B(x_0, R_0), \ \
\dist(x,K)\le 1,\hbox{ and
}0<r\le 10 r_0.
\end{equation}

As we proceed it might be
convenient to make the value of $r_0$ smaller (depending on the
constant $C_K$ in (\ref{eqn:7.2})), to make our estimates simpler. Without
loss of generality we may assume that $x_1=0\in \Sigma\cap B(x_0, R_0)$.
Let us recall the main
properties of $b$ and $Q$ that are used in this section. In particular we do
not need to know how $b$ and $Q$ are computed in terms of $\mu$ (see
(\ref{eqn:7.3}) and (\ref{eqn:7.4})). First,
$b=b_r\in\RR^{m}$ and
\begin{eqnarray}\label{eqn:8.4}
|b_r| & \le & \frac{n+2}{2\omega_nr^{n+2}}\int_{B(0,r)}r^2|y|d\mu(y)
 \le  \frac{(n+2)r}{2\omega_nr^n}\mu(B(0,r)), \\
|b_r|& \le & \frac{(n+2)r}{2}\left\{1+\frac{C_Kr^\alpha}{\omega_n}\right\}\le (n+2)r,
\nonumber
\end{eqnarray}
by (\ref{eqn:7.3}) and (\ref{eqn:7.5}), provided we assume that
$\frac{C_Kr^\alpha_0}{\omega_n}\le 1$. We do not explicitly need (\ref{eqn:8.4}), but the
homogeneity is important to keep in mind.

\noindent
Next, $Q$ is a quadratic form defined on $\RR^{m}$, (\ref{eqn:7.4}) and
(\ref{eqn:7.5}) ensure that for $x\in\RR^{m}$
\begin{eqnarray}\label{eqn:8.5}
0  \le  Q(x)& \le &\frac{n+2}{\omega_nr^{n+2}}\int_{B(0,r)}|x|^2r^2d\mu(y)
 \le  \frac{(n+2)|x|^2}{\omega_nr^n}\mu(B(0,r))  \\
& \le & (n+2)|x|^2(1+\omega^{-1}_nC_Kr^\alpha)\le (2n+4)|x|^2, \nonumber
\end{eqnarray}
and
\begin{equation}\label{eqn:8.6}
|Tr(Q)-n|\le CC_Kr^\alpha
\end{equation}
by (\ref{eqn:7.8}). It is convenient to set
\begin{equation}\label{eqn:8.7}
\widetilde Q(x)=|x|^2-Q(x).
\end{equation}
Then (\ref{eqn:7.9}) yields that
\begin{equation}\label{eqn:8.8}
|2\langle b_r,x\rangle - \widetilde Q(x)|\le Cr^{-1}|x|^3+CC_Kr^{2+\alpha}
\hbox{ for }x\in\Sigma\cap B(0,\frac{r}{2}).
\end{equation}
Initially we use (\ref{eqn:8.6}) and (\ref{eqn:8.8}) to derive more
information about $Q$ and $b$.
We work at scales of the form $\rho=r^{1+\gamma}$ smaller than $r$.
Here $\gamma$ is a positive constant that will assume several
different values.

It is
important to understand how (\ref{eqn:8.8}) is modified by a change of
scale. Set
\begin{equation}\label{eqn:8.9}
\Sigma_\rho=\frac{1}{\rho}\Sigma,
\end{equation}
and
\begin{equation}\label{eqn:8.10}
\Sigma'_\rho=\Sigma_\rho\cap B(0, \frac{r}{2\rho})=\frac{1}{\rho}(\Sigma\cap
B(0,\frac{r}{2})).
\end{equation}
Note that (\ref{eqn:8.3}) guarantees that we can choose an $n$-plane $L$
through the origin such that
\begin{equation}\label{eqn:8.21}
D[L\cap B(0,\rho), \Sigma\cap B(0,\rho)]\le\rho\theta(0,\rho)\le \rho\delta,
\end{equation}
where $D$ denotes the Hausdorff distance between sets, as in (\ref{eqn:1.3}).
(See also (\ref{eqn:1.2}) for the definition of $\theta(0,\rho)$).
Moreover for $z\in\Sigma'_\rho$ we can apply (\ref{eqn:8.8}) to $x=\rho z$
and get that
\begin{eqnarray}\label{eqn:8.11}
|2\langle \frac{b_r}{\rho}, z\rangle - \widetilde Q(z)| & = &
\rho^{-2}|2\langle b,x\rangle - \widetilde Q(x)| \\
& \le & C\rho^{-2}r^{-1}|x|^3 + CC_K\rho^{-2}r^{2+\alpha} \nonumber \\
& = & C\rho r^{-1}|z|^3 + CC_K\rho^{-2}r^{2+\alpha} \nonumber \\
& = & Cr^\gamma |z|^3 + CC_Kr^{\alpha-2\gamma} \nonumber
\end{eqnarray}
because $\rho=r^{1+\gamma}$. In particular,
\begin{equation}\label{eqn:8.12}
|\langle 2b_r r^{-1-\gamma}, z\rangle - \widetilde Q(z)| \le Cr^\gamma +
CC_Kr^{\alpha-2\gamma} =:\epsilon_0(r,\gamma)
\hbox{ for }z\in\Sigma_{r^{1+\gamma}}\cap B(0, \frac{1}{2}).
\end{equation}

To motivate the argument in the proof of Theorem {\ref{thm:1.17A}} we briefly
recall the main ideas in the proof of Theorem {\ref{thm:1.17}}.
(\ref{eqn:8.12}) encodes the information
required to estimate the quantity $\beta(0,\rho_0)$ defined in
(\ref{eqn:8.1}). In the codimension 1 case one needs to consider two cases.
Either $b$ in (\ref{eqn:8.12}) is very small, and then one obtains an
estimate on the smallest eigenvalue of $Q$ which allows one to say that
at the appropriate scale $\Sigma$ is very close to the plane normal to the
corresponding eigenspace. If $b$ is ``large'' then
at the appropriate scale $\Sigma$ is very close to the plane orthogonal to
$b$. In both cases one produces the normal vector which is orthogonal to the
plane $\Sigma$ is close to. In higher codimensions we need to produce
an $m-n$ orthonormal family of vectors whose span is orthogonal
to the $n$-plane $\Sigma$ is close
to, at a given scale.
The difficulty lies on the fact that there is only a single equation
at hand, namely (\ref{eqn:8.12}). To overcome this problem
we are forced to do a multiscale analysis of (\ref{eqn:8.12}).

Our first intermediate result is an estimate on $Q$ when $b_r$ is fairly
small.
Let us assume that
\begin{equation}\label{eqn:8.13}
|b|\le r^{1+2\theta}
\end{equation}
for some $\theta>0$. Then (\ref{eqn:8.12}) and (\ref{eqn:8.13}) ensure that
for $z\in\Sigma_{r^{1+\gamma}}\cap B(0,\frac{1}{2})$ we have
\begin{eqnarray}\label{eqn:8.14}
|\widetilde Q(z)| & \le & |\langle 2br^{-1-\gamma}, z\rangle| + Cr^\gamma +
CC_Kr^{\alpha-2\gamma} \\
& \le & r^{2\theta-\gamma} + Cr^\gamma + CC_Kr^{\alpha-2\gamma} \nonumber
\\
& =: & \epsilon_1(r,\theta,\gamma). \nonumber
\end{eqnarray}

Note that (\ref{eqn:8.14}) only provides useful information when
$\gamma$ satisfies
\begin{equation}\label{eqn:8.15}
0<\gamma<2\theta\;\;\hbox{ and }\;\;2\gamma<\alpha.
\end{equation}

Choose an orthonormal basis $(e_1, \ldots, e_m)$ of $\RR^{m}$ that
diagonalizes $Q$. Thus
\begin{equation}\label{eqn:8.16}
Q(z)=\sum^m_{i=1}\lambda_i\langle z,e_i\rangle^2
\end{equation}
for $z\in\RR^{m}$. Without loss of generality we may assume that
\begin{equation}\label{eqn:8.17}
\lambda_1\le\lambda_2\cdots\le \lambda_m.
\end{equation}
Note that $\lambda_1\ge 0$ because $Q(z)\ge 0$ (see (\ref{eqn:7.4}) or
(\ref{eqn:8.5})). Also, by (\ref{eqn:8.6})
\begin{equation}\label{eqn:8.18}
\sum^m_{i=1}\lambda_i=Tr(Q)\le n+CC_Kr^\alpha.
\end{equation}
In particular, by (\ref{eqn:8.17}) if $k=m-n$
\begin{equation}\label{eqn:8.19}
m\lambda_1\le Tr(Q)\le n+CC_Kr^\alpha<n+\frac{1}{2},
\end{equation}
\begin{equation}\label{eqn:8.19A}
(n+1)\lambda_k\le Tr(Q)\le n+CC_Kr^\alpha<n+\frac{1}{2},
\end{equation}
provided we take $r_0$ small enough. Thus
\begin{equation}\label{eqn:8.20}
0\le\lambda_1\le \frac{2n+1}{2m}\ \ \hbox{ and}\ \
\lambda_k\le\frac{2n+1}{2n+2}.
\end{equation}
This is just a crude first step. Our next goal is to obtain more precise
estimates on $Q$, when (\ref{eqn:8.13}) holds, i.e., $|b|\le r^{1+2\theta}$,
under the additional constraint that
\begin{equation}\label{eqn:8.38}
0<\theta<\frac{\alpha}{3}.
\end{equation}

\begin{lem}\label{lem:8.39}
Suppose that (\ref{eqn:8.13}), (\ref{eqn:8.15}) and (\ref{eqn:8.38}) hold.
Let $k=m-n$.
For $r_0$ small enough and $\epsilon_2(r,\theta,\gamma)=n
a^{-2}\epsilon_1(r,\theta,\gamma)$, where $a$ is a constant
that only depends on $n$ and $m$, we have
\begin{eqnarray}
0  \le  \sum_{i=1}^k\lambda_i &\le& \epsilon_2(r,\theta,\gamma) +
CC_Kr^{\alpha}, \label{eqn:8.40} \\
|\lambda_{k+i}-1| & \le & \epsilon_2(r,\theta,\gamma) +r^{\alpha/2}
\ \hbox{ for }\ 1\le i\le n,
\label{eqn:8.41}
\end{eqnarray}
and
\begin{equation}\label{eqn:8.42}
|\widetilde Q(z)-\sum_{l=1}^k\langle z,e_l\rangle^2|
\le \left(\epsilon_2(r,\theta,\gamma) + r^{\alpha/2}\right)
|z|^2\hbox{ for
}z\in\RR^{m}.
\end{equation}
\end{lem}

Note that (\ref{eqn:8.42}) automatically follows from (\ref{eqn:8.40}) and
(\ref{eqn:8.41}). In fact if we write
$z=\mathop{\sum}\limits^{m}_{i=1}z_ie_i$, then by (\ref{eqn:8.7}) and
(\ref{eqn:8.16}) we have
\begin{eqnarray}\label{eqn:8.43}
\widetilde Q(z)-\sum_{i=1}^k\langle z,e_i\rangle^2 & = & |z|^2-Q(z)-
\sum_{i=1}^k z^2_i \\
& = & \sum^n_{i=1}z^2_{k+i}-\sum^m_{i=1}\lambda_iz^2_i \nonumber \\
& = & -\sum^k_{l=1}\lambda_lz^2_{l} + \sum^n_{i=1}(1-\lambda_{i+k})z^2_i.
\nonumber
\end{eqnarray}

Note that the choice $\gamma=\theta$ with $\theta$ as in (\ref{eqn:8.14})
satisfies (\ref{eqn:8.15}). In this case
Lemma \ref{lem:8.39} becomes

\begin{cor}\label{cor:8.39B}
Suppose that (\ref{eqn:8.13}), and (\ref{eqn:8.38}) hold.
For $r_0$ small enough
\begin{eqnarray}
0  \le  \sum_{i=1}^k\lambda_i &\le& Cr^\theta, \label{eqn:8.40B} \\
|\lambda_{k+i}-1| & \le & Cr^{\theta}
\ \hbox{ for }\ 1\le i\le n,
\label{eqn:8.41B}
\end{eqnarray}
and
\begin{equation}\label{eqn:8.42B}
|\widetilde Q(z)-\sum_{l=1}^k\langle z,e_l\rangle^2|
\le Cr^\theta
|z|^2\hbox{ for
}z\in\RR^{m},
\end{equation}
where $C$ is a constant that depends on $K$, $n$ and $m$.
\end{cor}

To prove Lemma \ref{lem:8.39} we need some preliminary results. The first one is the following.

\begin{lem}\label{lem:8.22}
Let $L$ denote an $n$-plane satisfying
(\ref{eqn:8.21}). For $l=1,\cdots, k$ let $v_l$ denote the orthogonal
projection of $e_l$ onto $L$.
If $\delta$ and $r_0$ are chosen small enough,
\begin{equation}\label{eqn:8.23}
|\sum_{l=1}^k x_l v_l-\sum_{l=1}^k x_l e_l|\ge C^{-1},
\end{equation}
whenever $\sum_{l=1}^k |x_l|^2=1$
\end{lem}

\noindent In Lemma \ref{lem:8.22}, $\delta$ and $r_0$ depend on $n$,
$\alpha$, $\theta$ and $\gamma$. At most $2k$ values
of $\theta$ and $\gamma$, are used depending only on
$\alpha$, and a choice of $\theta$. Thus one can always choose
$\delta>0$ and $r_0>0$ to work simultaneously for all our choices.
The constant $C>1$ depends only on $n$ and $m$.

\begin{proof}
To prove Lemma \ref{lem:8.22}, we first estimate $\widetilde Q(z)$ for $z\in
L\cap B(0,\frac{1}{3})$. Since $\rho z\in L\cap B(0,\rho)$, where $\rho=r^{1+\gamma}$, (\ref{eqn:8.21}) guarantees that there is a point $x\in\Sigma\cap
B(0,\rho)$ such that $|x-\rho z|\le 2\rho\delta$. If $\delta$ is small
enough, $|\rho^{-1}x|<\frac{1}{2}$, and so (\ref{eqn:8.14}) ensures that
$|\widetilde Q(\rho^{-1}x)|\le \epsilon_1(r,\theta,\gamma)$.
Also, $|\rho^{-1}x-z|=\rho^{-1}|x-\rho z|\le 2\delta$, and hence (\ref{eqn:8.5}) and (\ref{eqn:8.7})
guarantee that
\begin{equation}\label{eqn:8.24}
|\widetilde Q(\rho^{-1}x)-\widetilde Q(z)|\le C\delta.
\end{equation}
Altogether,
\begin{equation}\label{eqn:8.25}
|\widetilde Q(z)|\le \epsilon_1(r,\theta,\gamma)+
C\delta\hbox{ for }z\in L\cap
B(0,\frac{1}{3}).
\end{equation}
We are now ready to prove (\ref{eqn:8.23}).
Let $u=\sum_{l=1}^k x_l e_l$ with $|u|=1$. Then $w=\sum_{l=1}^k x_l v_l$
satisfies $|w|\le 1$, thus (\ref{eqn:8.25})
guarantees that
\begin{equation}\label{eqn:8.26}
|\widetilde Q(u)| \le  |\widetilde Q(w)|+|\widetilde Q(u)-\widetilde
Q(v)|
 \le  \epsilon_1(r,\theta, \gamma)+C\delta +
|\widetilde Q(u)-\widetilde Q(w)|.
\end{equation}
On the other hand since $u$ belongs to the span of the first $k$ eigenvectors
of $Q$ we have that
\begin{equation}\label{eqn:8.27}
\widetilde Q(u)=1-Q(u)\ge 1-\lambda_k\ge \frac{1}{2n+2}
\end{equation}
by (\ref{eqn:8.7}), (\ref{eqn:8.16}), and (\ref{eqn:8.20}). If $\delta$ and
$r_0$ are small enough, (\ref{eqn:8.26}) and (\ref{eqn:8.27}) imply that
\begin{equation}\label{eqn:8.28}
|\widetilde Q(u)-\widetilde Q(w)|\ge \frac{1}{4n+4}.
\end{equation}
Thus $w$ cannot be too close to $u$ (because of (\ref{eqn:8.5})), and
(\ref{eqn:8.23}) holds. \qed
\end{proof}

\noindent Now we want to use the fact that $\Sigma\cap B(x_0, R_0)$
is Reifenberg flat with
vanishing constant to get important topological information on
$\Sigma_\rho\cap B(0, \frac{1}{2})$, $\rho=r^{1+\gamma}$. Denote by $P$
the $n$-plane through 0 which is orthogonal to $e_1,\cdots, e_k$. Thus
\begin{equation}\label{eqn:8.29}
P={\rm span}^\perp(e_1, \cdots, e_k)=\mathrm{span}(e_{k+1},\ldots, e_m).
\end{equation}

\noindent
Call $\pi$ the orthogonal projection onto $P$.
Also denote by $\pi^\ast:
\RR^m\to L$ the projection onto $L$ parallel to the direction
$\{e_1,\cdots, e_k\}$, i.e.
$$
\pi^{\ast}(x)=\pi^{\ast}(\sum_{l=1}^m x_l e_l)=
\sum_{l=1}^n x_{k+l} e_{k+l}+\sum_{l=1}^ky_l e_l
$$
where the orthogonal projection of $\sum_{l=1}^ky_l e_l$ into $L^\perp$
coincides with that of $\sum_{l=1}^n x_{k+l} e_{k+l}$.
Here $L$ is as in (\ref{eqn:8.21}), and $L^\perp$ denotes the
$(m-n)$ space orthogonal to $L$. Denote by $\pi'$ the orthogonal projection
of $\RR^m$ onto $L^\perp$.
Lemma \ref{lem:8.22} ensures that
\begin{eqnarray}\label{eqn:8.30A}
C^{-1}|\sum_{l=1}^ky_l e_l| &\le & |\sum_{l=1}^ky_l e_l -\sum_{l=1}^ky_l v_l|
=|\pi'(\sum_{l=1}^ky_l e_l)|\\
&\le & |\pi'(\sum_{l=1}^n x_{k+l} e_{k+l})|\le |\sum_{l=1}^n x_{k+l} e_{k+l}|
\le |x|.\nonumber
\end{eqnarray}
Thus
\begin{equation}\label{eqn:8.30}
|\pi^\ast(x)|\le C_0|x|\hbox{ for }x\in\RR^{m}.
\end{equation}
Here $C_0= 2C$ where $C$ is as in (\ref{eqn:8.23}), a constant that
depends only on $n$ and $m$.

Set $a=(4C_0)^{-1}$,
and recall that $\rho=r^{1+\gamma}$, where $\gamma$ satisfies
(\ref{eqn:8.15}). The same argument as in \cite{DKT} guarantees that:

\begin{figure}\label{fig:8.1}
\hskip1.5in\begin{psfrags}
\includegraphics[width=4in]{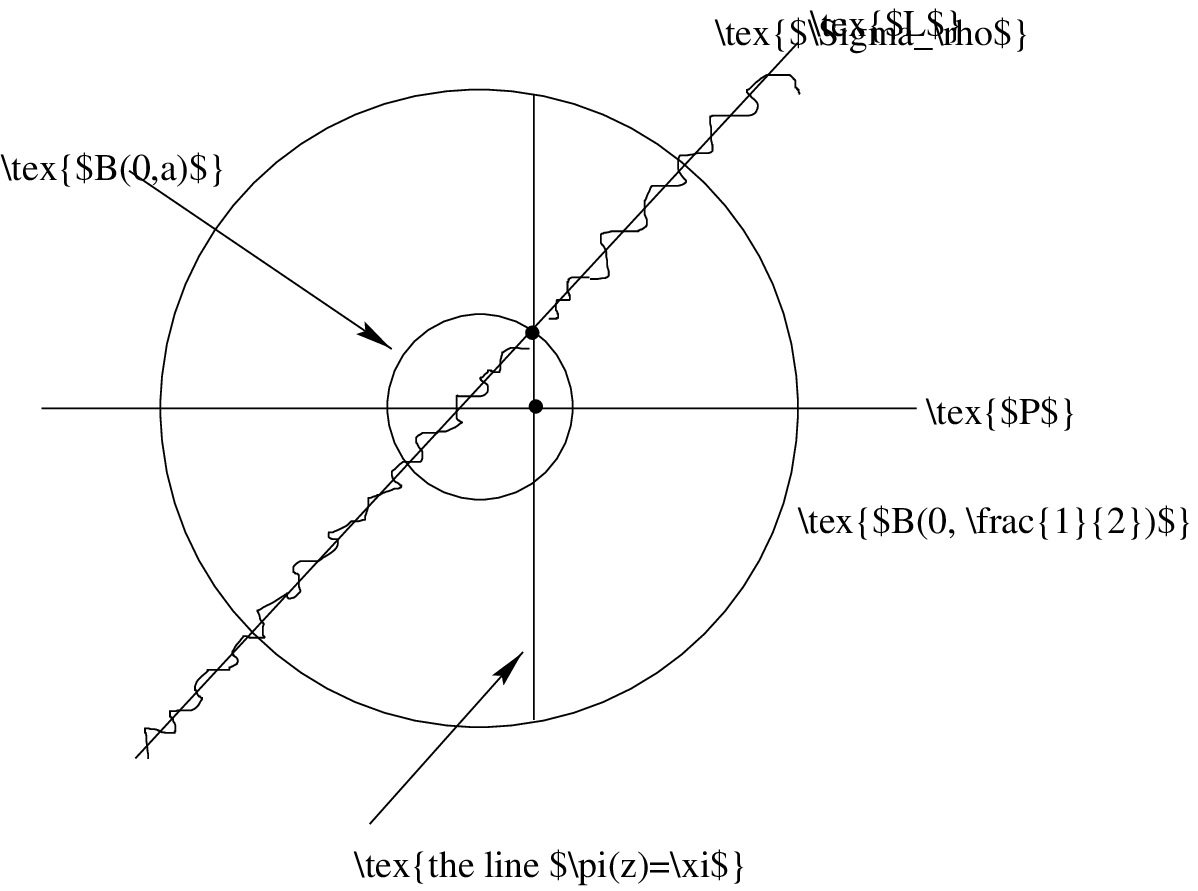}
\end{psfrags}
\caption{}
\end{figure}

\begin{lem}\label{lem:8.32}
For every $\xi\in P\cap\overline B(0,a)$, there is a point $z\in\Sigma_\rho
\cap B(0, \frac{1}{2})$ such that $\pi(z)=\xi$. {\rm [See Figure 8.1]}
\end{lem}

\noindent We have gathered all the information needed to prove Lemma \ref{lem:8.39}.
\medskip

\noindent{\bf Proof of Lemma \ref{lem:8.39}}: To prove (\ref{eqn:8.40}) and
(\ref{eqn:8.41}), we apply Lemma \ref{lem:8.32}
with $\xi=ae_{k+i}$, $1\le i\le n$. We choose $\gamma$ so that
(\ref{eqn:8.15}) holds. We get that for some $(t_i^1,\cdots, t_i^k)\in\RR^k$,
\begin{equation}\label{eqn:8.44}
z_i=\sum_{l=1}^k t_i^le_l+ae_{k+i}\in\Sigma_\rho\cap B(0, \frac{1}{2}).
\end{equation}
If we take $\gamma=\theta$, (\ref{eqn:8.14}) and (\ref{eqn:8.38}) guarantee that
\begin{equation}\label{eqn:8.45}
|\widetilde Q(z_i)|\le \epsilon_1(r,\theta,\gamma)
\end{equation}
Combining (\ref{eqn:8.7}), (\ref{eqn:8.16}) and (\ref{eqn:8.44}) we obtain
that
\begin{equation}\label{eqn:8.46}
\widetilde Q(z_i)  =  |z_i|^2-Q(z_i)
 =  \sum_{l=1}^k(1-\lambda_l)(t_i^l)^2 + (1-\lambda_{k+i})a^2.
\end{equation}
Since $1-\lambda_l\ge (2n+2)^{-1}$ for $1\le l\le k$
(by (\ref{eqn:8.20})), we get that
\begin{equation}\label{eqn:8.47}
(1-\lambda_{k+i})a^2\le\widetilde Q(z_i)\le \epsilon_1(r,\theta,\gamma)
\end{equation}
(by (\ref{eqn:8.45})). Thus
\begin{equation}\label{eqn:8.48}
\lambda_{k+i}\ge 1-a^{-2}\epsilon_1(r,\theta,\gamma),
\end{equation}
for $1\le i\le n$, and hence
\begin{equation}\label{eqn:8.49}
\sum^n_{i=1}\lambda_{k+i}\ge n-\epsilon_2(r,\theta,\gamma).
\end{equation}
By (\ref{eqn:8.18}) and (\ref{eqn:8.49}) we have that
\begin{equation}\label{eqn:8.50}
\sum_{l=1}^k\lambda_l=Tr(Q)-\sum^n_{i=1}\lambda_i\le CC_Kr^\alpha+
a^{-2}\epsilon_1(r,\theta,\gamma)
\end{equation}
This proves (\ref{eqn:8.40}), because we already know that
$\sum_{l=1}^k\lambda_l\ge 0$.
To prove (\ref{eqn:8.41}), we proceed by contradiction and suppose that we
can find $1\le i_0\le n$ such that
\begin{equation}\label{eqn:8.51}
\lambda_{k+i_0}>1+\epsilon_2(r,\theta,\gamma)+r^{\alpha/2}.
\end{equation}
Then (\ref{eqn:8.48}) and
(\ref{eqn:8.51}) yield
\begin{equation}\label{eqn:8.52}
\sum^m_{i=1}\lambda_i  \ge  \sum^n_{i=1}\lambda_{k+i}
 \ge  \lambda_{k+i_0} + (n-1)(1-\frac{\epsilon_2(r,\theta,\gamma)}{n})
 >  n + r^{\alpha/2}
\end{equation}
This contradicts (\ref{eqn:8.18}), thus (\ref{eqn:8.51}) is impossible and
(\ref{eqn:8.41}) holds. We already observed earlier that (\ref{eqn:8.42}) is a consequence of
(\ref{eqn:8.40}) and (\ref{eqn:8.41}), and so
Lemma \ref{lem:8.39}
follows. \qed
\medskip

Next we use Corollary \ref{cor:8.39B} to rewrite (\ref{eqn:8.12}), still under
the assumption that (\ref{eqn:8.13}) holds for some $\theta\in (0,
\frac{\alpha}{3})$. Combining (\ref{eqn:8.12}) and (\ref{eqn:8.42}) we get
that for $z\in \Sigma_\rho\cap B(0, \frac{1}{2})$
\begin{eqnarray}\label{eqn:8.53}
|\langle 2br^{-1-\gamma}, z\rangle-
\sum_{l=1}^k\langle z,e_l\rangle^2|
 & \le &
Cr^\gamma + CC_Kr^{\alpha-2\gamma}+ |\widetilde Q(z)-
\sum_{l=1}^k\langle z,e_l\rangle^2| \\
& \le & Cr^\gamma+CC_Kr^{\alpha-2\gamma} + Cr^\theta \nonumber\\
& =: & \epsilon_3(r,\theta,\gamma).\nonumber
\end{eqnarray}
Note that here $\rho=r^{1+\gamma}$ for any $\gamma>0$ (as in
(\ref{eqn:8.12})). Of course (\ref{eqn:8.53}) only provides useful
information when $0<\gamma<\frac{\alpha}{2}$.

Next we want to get a better estimate on the ``tangential part'' of b.
This allows us to estimate $\beta(0,s)$ as defined in (\ref{eqn:8.1})
for an appropriately chosen $s$.

\begin{prop}\label{prop:8.100}
If $b=\mathop{\sum}\limits^{m}_{i=1} b_ie_i$, then
\begin{equation}\label{eqn:8.54}
|b_{k+i}|\le C r^{1+\eta}\epsilon_3(r,\theta,\eta) + Cr^{1+4\theta-\eta}
\hbox{ for }1\le i\le n.
\end{equation}
\end{prop}

\begin{rem} The goal is to show that given appropriate choices
for $\theta$ and $\eta$ satisfying
(\ref{eqn:8.38}) and (\ref{eqn:8.15}) with $\eta$ in place of $\gamma$,
(\ref{eqn:8.54}) provides an improvement over (\ref{eqn:8.13}).
In the codimension 1 case it was possible to choose $\gamma=\eta=3\theta/2$.
The reader will note that this choice does improve estimate
(\ref{eqn:8.13}). Unfortunately in the higher codimension set up it is
premature to choose $\eta$ at this stage.

\end{rem}

\begin{proof}
Choose $\theta$ and $\gamma=\eta$ such that (\ref{eqn:8.38})
and (\ref{eqn:8.15})
hold. We can then apply Lemma
\ref{lem:8.32}. Fix $i\in\{1, 2, \ldots, n\}$ and apply Lemma \ref{lem:8.32}
to the two points $\xi_\pm=\pm ae_{k+i}$. We get $k$ vectors
$(t_1^\pm, \cdots, t_k^\pm)$ such that
\begin{equation}\label{eqn:8.55}
z_\pm = \sum_{l=1}^k t_l^\pm e_l \pm ae_{k+i} \in
\Sigma_{r^{1+\eta}}\cap B(0, \frac{1}{2}).
\end{equation}
Then (\ref{eqn:8.53})
implies that
\begin{equation}\label{eqn:8.56}
\sum_{l=1}^k 2b_lr^{-1-\eta}t_l^\pm \pm 2b_{k+i}r^{-1-\eta} a-
\sum_{l=1}^k(t_l^\pm)^2 \ge
-\epsilon_3(r,\theta,\eta).
\end{equation}
Set $f_l(t)=2b_lr^{-1-\eta}t-t^2$ for $1\le l\le k$. Then
\begin{equation}\label{eqn:8.57}
f_l(t) = (b_lr^{-1-\eta})^2 - (b_lr^{-1-\eta}-t)^2 \le
(b_lr^{-1-\eta})^2
\end{equation}
for all $t\in\RR$. Hence by (\ref{eqn:8.56}) and (\ref{eqn:8.57}) we have
that
\begin{equation}\label{eqn:8.58}
\pm 2b_{k+i}r^{-1-\eta} a \ge -\epsilon_3(r,\theta,\eta)-
\sum_{l=1}^kf_l(t_l^\pm) \ge
-\epsilon_3(r,\theta,\eta) - \sum_{l=1}^k(b_lr^{-1-\eta})^2.
\end{equation}
Here we have two inequalities, one for each sign $\pm$. Thus by
(\ref{eqn:8.13})
\begin{equation}\label{eqn:8.59}
|b_{k+i}| \le (2a)^{-1}r^{1+\eta}\epsilon_3(r,\theta,\eta) +
(2a)^{-1}|b|^2r^{-1-\eta} \le Cr^{1+\eta}\epsilon_3(r,\theta,\eta) +
Cr^{1+4\theta-\eta}.
\end{equation} \qed
\end{proof}
\medskip

Combining (\ref{eqn:8.54}) and (\ref{eqn:8.53}), we get that for
$z\in\Sigma_\rho\cap B(0, \frac{1}{2})$, where $\rho=r^{1+\gamma}$,
\begin{eqnarray}\label{eqn:8.60}
& & \\
|\langle 2\sum_{l=1}^k b_lr^{-1-\gamma}, z\rangle-
\sum_{l=1}^k\langle z,e_l\rangle^2|
& \leq & \epsilon_3(r,\theta,\gamma) +|\sum_{i=1}^n 2b_{k+i}r^{-1-\gamma}\langle
z, e_{k+i}\rangle|\nonumber\\
&\le &  \epsilon_3(r,\theta,\gamma) +Cr^{-1-\gamma}
r^{1+\eta}\epsilon_3(r,\theta,\eta) +
Cr^{4\theta-\eta-\gamma}\nonumber\\
&\le& C(r^{\gamma}+r^{\alpha-2\gamma}+r^\theta +r^{2\eta-\gamma}+
r^{\alpha-\eta-\gamma} + r^{\theta+\eta-\gamma} + r^{4\theta-\eta-\gamma}).
\nonumber
\end{eqnarray}

This holds for $\theta$ as in (\ref{eqn:8.38}), $\eta$ satisfying
\begin{equation}\label{eqn:8.15E}
0<\eta<2\theta\;\;\hbox{ and }\;\;2\eta<\alpha,
\end{equation}
and all $\gamma>0$ as in (\ref{eqn:8.53})
It only provides an interesting estimate for some values of $\gamma$.
Choose
\begin{equation}\label{eqn:800}
0<4\gamma<\alpha,
\end{equation}
and define
\begin{equation}\label{eqn8:60E}
\epsilon_4(r,\theta,\gamma,\eta):=C
(r^{\gamma}+r^\theta +r^{2\eta-\gamma}
+ r^{\theta+\eta-\gamma} + r^{4\theta-\eta-\gamma}).
\end{equation}
Then (\ref{eqn:8.60}) becomes
\begin{equation}\label{eqn:8.60EE}
|\langle 2\sum_{l=1}^k b_lr^{-1-\gamma}, z\rangle-
\sum_{l=1}^k\langle z,e_l\rangle^2|\le \epsilon_4(r,\theta,\gamma,\eta).
\end{equation}

\begin{prop}\label{prop:8.101}
With the notation above we have that
\begin{equation}\label{eqn:8.61}
|\sum_{l=1}^k\langle z,  e_l \rangle e_l|\le
3\epsilon_4(r,\theta,\gamma,\eta)^{1/2}
\hbox{ for }
z\in\Sigma_\rho\cap B(0, \frac{1}{4}).
\end{equation}
Here $\rho=r^{1+\gamma}$. The exponents $\theta$, $\gamma$ and $\eta$ satisfy
(\ref{eqn:8.38}), (\ref{eqn:8.15E}) and (\ref{eqn:800}).
\end{prop}

\begin{proof}
Set $z^{\perp}=\sum_{l=1}^k\langle z, e_l\rangle e_l$ for
$z\in\Sigma_\rho\cap B(0, \frac{1}{2})$.
Then (\ref{eqn:8.60}) can be written
\begin{equation}\label{eqn:8.62}
|z^{\perp}(z^{\perp}-d)|\le \epsilon_4(r,\theta,\gamma,\eta),
\end{equation}
where $d=2b^{\perp}r^{-1-\gamma}$. This forces
\begin{equation*}\label{eqn:8.63a}
|z^{\perp}|\le\epsilon_4(r,\theta,\gamma,\eta)^{\frac{1}{2}}\tag{8.59+}
\end{equation*}
or
\begin{equation*}\label{eqn:8.63b}
|z^{\perp}-d|\le\epsilon_4(r,\theta,\gamma,\eta)^{\frac{1}{2}}.\tag{8.59-}
\end{equation*}
If $|d|\le 2\epsilon_4(r,\theta,\gamma,\eta)^{\frac{1}{2}}$, then (\ref{eqn:8.61})
trivially follows from this. So let us assume that
$|d|>2\epsilon_4(r,\theta,\gamma,\eta)^{\frac{1}{2}}$. Denote by $\U$ the connected
component of $\Sigma_\rho\cap B(0, \frac{1}{2})$ containing the origin, and
set
\begin{equation}\label{eqn:8.64}
\U_\pm = \{z\in\U; \hbox{(8.59$\pm$) holds}\}.
\end{equation}
Obviously $\U_+$ and $\U_-$ are closed in $\U$, and since
$\U$ is the disjoint union of $\U_+$ and $\U_-$ (because
$|d|>2\epsilon_4(r,\theta,\gamma,\eta)^{\frac{1}{2}}$),
$\U$ must be equal to $\U_+$. Thus
to prove
(\ref{eqn:8.61}) it is enough to show that
\begin{equation}\label{eqn:8.65}
\Sigma_\rho\cap B(0, \frac{1}{4})\subset\U.
\end{equation}
Since (\ref{eqn:8.3}) holds, $\Sigma_\rho$
is locally Reifenberg flat and the same argument used in
Section 8 of \cite{DKT}
yields (\ref{eqn:8.65}). Proposition \ref{prop:8.101} follows. \qed
\end{proof}

\noindent Note that (\ref{eqn:8.61}) says that if $|b_r|\le
r^{1+2\theta}$, then
\begin{equation}\label{eqn:8.68}
 \beta(0, \frac{1}{4}r^{1+\gamma}) \le
12 \epsilon_4(r,\theta,\gamma,\eta)^{\frac{1}{2}}=:
\epsilon_5(r,\theta,\gamma,\eta),
\end{equation}
where $\beta(0,s)$ is defined as in (\ref{eqn:8.1})
and
\begin{equation}\label{eqn:8.68A}
\epsilon_4(r,\theta,\gamma,\eta)
=C
(r^{\gamma}+r^\theta +r^{2\eta-\gamma}
+ r^{\theta+\eta-\gamma} + r^{4\theta-\eta-\gamma}).
\end{equation}
Here $C$ depends on $n$, $m$ and $K$.
The estimate (\ref{eqn:8.68}) holds for all
exponents $\theta$, $\gamma$ and $\eta$ satisfying
(\ref{eqn:8.38}), (\ref{eqn:8.15E}) and (\ref{eqn:800}).

Recall that $b_r=b$. So far we have omitted the dependence of $r$ to simplify
the notation, as there was no room for confusion. From now on we need
to keep track of it as it will be made clear shortly.

When (\ref{eqn:8.13}) does not hold, i.e.
\begin{equation}\label{eqn:8.69}
|b|>r^{1+2\theta},
\end{equation}
(\ref{eqn:8.12}) and (\ref{eqn:8.5}) tell us that
\begin{equation}\label{eqn:8.70}
|\langle 2br^{-1-\gamma}, z\rangle| \le |\widetilde Q(z)|+Cr^\gamma+CC_K
r^{\alpha-2\gamma}\le C
\end{equation}
for $z\in\Sigma_{r^{1+\gamma}}\cap B(0, \frac{1}{2})$, provided that we choose
$0<\gamma<\frac{\alpha}{2}$, $r<r_0$ and $r_0$ small enough.
Set $\tau=|b|^{-1}b$. Then
\begin{equation}\label{eqn:8.71}
|\langle \tau,z\rangle|\le C|b|^{-1}r^{1+\gamma}\le Cr^{\gamma-2\theta}
\end{equation}
for $z\in\Sigma_{r^{1+\gamma}}\cap B(0, \frac{1}{2})$.
In the codimension 1 case $|\langle \tau,z\rangle|$
measures the distance from $z$ to the $n$-plane orthogonal to $\tau$.
(\ref{eqn:8.71}) implies that
$\beta(0, \frac{1}{4}r^{1+\gamma})\le Cr^{\gamma-2\theta}.$
In this case choosing $\eta$, $\gamma$, $\theta$ appropriately
one can guarantee that
$\beta(0, \frac{1}{4}r^{1+\gamma})$ is bounded by a positive
power of $r$. This case is done in \cite{DKT}.

In codimension
$k=m-n$ we need to
produce a $k$ plane such that $z^{\perp}$, the orthogonal projection
$z\in\Sigma_{r^{1+\gamma}}\cap B(0, \frac{1}{2})$ onto this plane, is
bounded by a positive power on $r$.
To accomplish this we need to choose
$3k$ exponents $\eta_i$, $\gamma_i$, $\theta_i$ and $k+1$ radii $r_i$
with $1\le i\le k$
satisfying
\begin{equation}\label{eqn:001}
0<3\theta_i<\alpha,\ \
0<4\gamma_i<\alpha,\ \ 0<\eta_i<2\theta_i\ \hbox{ and }\ 2\eta_i<\alpha,
\end{equation}
and
\begin{equation}\label{eqn:002}
r_1=r,\ \ \ \ r_{i+1}=r_i^{1+\gamma_i}.
\end{equation}
The difficulty lies on the fact that several additional
compatibility conditions arise along the proof, and we need to check that
they can be satisfied.

\begin{case}
There exists $\,i=1,\ldots, k=m-n$ such that
\begin{equation}\label{eqn:003}
|b_{r_i}|\le r_i^{1+2\theta_i},
\end{equation}
then (\ref{eqn:8.68}) ensures that
\begin{equation}\label{eqn:004}
\beta(0,\frac{1}{4}r_{i+1})\le {\epsilon_5(r_i, \theta_i,
\gamma_i, \eta_i)}.
\end{equation}
Thus we have
\begin{equation}\label{eqn:004A}
\beta(0,\frac{r_{k+1}}{4})\le \frac{r_{i+1}}{r_{k+1}}
\beta(0,\frac{r_{i+1}}{4})\le \frac{r_{i+1}}{r_{k+1}}
{\epsilon_5(r_i, \theta_i, \gamma_i, \eta_i)}.
\end{equation}
\end{case}

\begin{case}
For all $i=1,\cdots, k=m-n$
\begin{equation}\label{eqn:005}
|b_{r_i}|\ge r_i^{1+2\theta_i},
\end{equation}
then (\ref{eqn:8.71}) guarantees that
\begin{equation}\label{eqn:006}
|\langle\tau_i,z\rangle|\le C r^{\gamma_i-2\theta_i}_{i}\ \ \hbox{ for }\ \
z\in\Sigma_{r_{i+1}}\cap B(0,\frac{1}{2}),\ \ \hbox{ where }\ \
\tau_i=\frac{b_{r_i}}{|b_{r_i}|}.
\end{equation}
Thus
\begin{equation}\label{eqn:007}
|\langle\tau_i,x\rangle|\le Cr_{i+1} r^{\gamma_i-2\theta_i}_i =
Cr^{1+2\gamma_i-2\theta_i}_i
\ \ \hbox{ for }\ \
x\in\Sigma\cap B(0,\frac{r_{i+1}}{2}).
\end{equation}
If $j\ge i+1$ then $r_j\le r_{i+1}$. Using the definition of $b_{r_j}$ which
appears in (\ref{eqn:7.3}) we obtain from (\ref{eqn:007}) that
\begin{equation}\label{eqn:008}
|\langle\tau_i, b_{r_j}\rangle|\le C r_{i+1} r^{\gamma_i-2\theta_i}_i.
\end{equation}
The definition of $\tau_j$ combined with (\ref{eqn:005}) and (\ref{eqn:008})
yield
\begin{equation}\label{eqn:009}
|\langle\tau_i, \tau_{j}\rangle|\le C r^{-1-2\theta_j}_j
r_{i+1} r^{\gamma_i-2\theta_i}_i = C r^{-1-2\theta_j}_j
r^{1+2\gamma_i-2\theta_i}_i\ \ \hbox{ for }\ \ j\ge i+1.
\end{equation}
\end{case}

Our goal is to show that in either case there exists $s$, a power of $r$,
such that $\beta(0,s)$ is bounded above by a power of $r$ (i.e.\ of $s$).
In Case 1 it suffices to show that the exponent of $r$ in the right hand side
of (\ref{eqn:004A}) is positive. In Case 2 we first need to show that
the vectors $\tau_l$ for $l=1,\cdots, k$ are linearly independent
(in fact almost orthogonal). This is achieved by showing that the exponent
of $r$ that appears in (\ref{eqn:009}) can be made positive. Once we know that
the vectors $\tau_l$ for $l=1,\cdots, k$ are almost orthogonal
(\ref{eqn:007}) provides an estimate for $\beta(0,\frac{r_{k+1}}{4})$.
In fact assume that $|\langle\tau_i, \tau_{j}\rangle|<\frac{1}{2k}$ for
$i,\ j= 1,\cdots, k$, $i\not = j$. Then (\ref{eqn:007}) yields that
$x^{\perp}$ the orthogonal projection
of $x\in \Sigma\cap B(0,\frac{r_{k+1}}{2})$ satisfies
\begin{equation}\label{eqn:010}
|x^{\perp}|\le C\max_{1\le i\le k}r^{1+2\gamma_i-2\theta_i}_i.
\end{equation}
Therefore
\begin{equation}\label{eqn:011}
\beta(0,\frac{r_{k+1}}{4})\le C r_{k+1}^{-1}
\max_{1\le i\le k}r^{1+2\gamma_i-2\theta_i}_i,
\end{equation}
where $C$ is a constant that depends on $n$, $m$ and $K$.

Our immediate task is to show that by choosing $\theta_i$, $\eta_i$ and
$\gamma_i$ appropriately and satisfying (\ref{eqn:001}) the right
hand sides of (\ref{eqn:004A}),(\ref{eqn:009}), and (\ref{eqn:011})
can be written as positive powers of $r$.

We first focus on the right hand side of (\ref{eqn:009}) for $j\ge i+1$.
Recall that
\begin{equation}\label{eqn:012}
r_j=r^{1+\gamma_{j-1}}_{j-1} = r^{\prod^{j-1}_{l=i}(1+\gamma_l)}_i,
\end{equation}
hence
\begin{equation}\label{eqn:013}
r^{-1-2\theta_j}_j
r^{1+2\gamma_i-2\theta_i}_i =
r^{1+2\gamma_i-2\theta_i-(1+2\theta_j){\prod^{j-1}_{l=i}(1+\gamma_l)}}_i.
\end{equation}
Thus for each $i=1, \cdots, k$ and $j\ge i+1$ we need
\begin{equation}\label{eqn:014}
1+2\gamma_i-2\theta_i-(1+2\theta_j){\prod^{j-1}_{l=i}(1+\gamma_l)}>0
\end{equation}
Similarly the right hand side of (\ref{eqn:011}) yields
\begin{equation}\label{eqn:015}
r_{k+1}^{-1}
r^{1+2\gamma_i-2\theta_i}_i =
r_i^{1+2\gamma_i-2\theta_i- \prod^{k}_{l=i}(1+\gamma_l)},
\end{equation}
which leads to the condition
\begin{equation}\label{eqn:016}
1+2\gamma_i-2\theta_i- \prod^{k}_{l=i}(1+\gamma_l)>0,
\end{equation}
for all $i=1,\cdots, k$.

Note that if (\ref{eqn:014}) is satisfied for $j=k+1$ then so is
(\ref{eqn:016}). Moreover (\ref{eqn:014}) applied to $j=i+1$ requires
that for $i=1,\cdots, k$
\begin{equation}\label{eqn:016A}
\gamma_i > 2\theta_i.
\end{equation}

The right hand side of (\ref{eqn:004A}) produces five conditions for each
$i=1,\cdots, k$. In fact the term
\begin{equation}\label{eqn:016B}
r^{-1}_{k+1} r_{i+1}=r^{1+\gamma_i-\prod^{k}_{l=i}(1+\gamma_l)}_i
\end{equation}
is multiplied by each one of the terms in
$\epsilon_5(r_i,\theta_i,\gamma_i, \eta_i)$. We obtain:
\begin{equation}\label{eqn:017}
1+\frac{3}{2}\gamma_i-\prod^{k}_{l=i}(1+\gamma_l)>0,
\end{equation}
\begin{equation}\label{eqn:018}
1+\gamma_i+\frac{\theta_i}{2}-\prod^{k}_{l=i}(1+\gamma_l)>0,
\end{equation}
\begin{equation}\label{eqn:019}
1+\frac{\gamma_i}{2}+\eta_i-\prod^{k}_{l=i}(1+\gamma_l)>0,
\end{equation}
\begin{equation}\label{eqn:020}
1+\frac{\gamma_i}{2}+\frac{\theta_i}{2}+\frac{\eta_i}{2}
-\prod^{k}_{l=i}(1+\gamma_l)>0,
\end{equation}
\begin{equation}\label{eqn:021}
1+\frac{\gamma_i}{2}+2{\theta_i}-\frac{\eta_i}{2}
-\prod^{k}_{l=i}(1+\gamma_l)>0.
\end{equation}

Using (\ref{eqn:016A}) and (\ref{eqn:001}) we observe that
\begin{equation}\label{eqn:022}
1+\frac{3}{2}\gamma_i-\prod^{k}_{l=i}(1+\gamma_l)
\ge 1+\frac{\gamma_i}{2}+2{\theta_i}-\frac{\eta_i}{2}
-\prod^{k}_{l=i}(1+\gamma_l),
\end{equation}
and
\begin{equation}\label{eqn:023}
1+\gamma_i+\frac{\theta_i}{2}-\prod^{k}_{l=i}(1+\gamma_l)\ge
1+\frac{\gamma_i}{2}+\frac{\theta_i}{2}+\frac{\eta_i}{2}
-\prod^{k}_{l=i}(1+\gamma_l).
\end{equation}
Thus (\ref{eqn:017}) and (\ref{eqn:018}) are satisfied
whenever (\ref{eqn:001}), (\ref{eqn:016}), (\ref{eqn:020}) and (\ref{eqn:021})
hold.

At this point we are ready to choose the form of the exponents. Let
\begin{equation}\label{eqn:024}
\gamma_{l+1}=\kappa\gamma_l,\ \ \theta_{l+1}=\kappa\theta_l,\ \
\eta_{l+1}=\kappa\eta_l,
\end{equation}
with
\begin{equation}\label{eqn:025}
0<\kappa<\frac{1}{16},\ \ 0 <3\theta_1<\alpha,\ \
0<4\gamma_1<\alpha,\ \ \ \hbox{ and }\ \
0<\eta_1=\frac{3}{2}\theta_1<2\theta_1.
\end{equation}
Note that this implies that $2\eta_1<\alpha$.

This choice of $\gamma_1,\ \theta_1, \ \eta_1$ and $\kappa$ ensure
that (\ref{eqn:001}) is satisfied, that $\eta_i= \frac{3}{2}\theta_i$,
and that $\gamma_i<1/4$.

Note that three of the four remaining conditions (\ref{eqn:014}),
(\ref{eqn:019}), (\ref{eqn:020}) and (\ref{eqn:021}) contain the term
$\prod^{k}_{l=i}(1+\gamma_l)$, or a product term which is bounded by it.
Using the fact that for $x\ge 0$ $1+x\le e^x$ and that for $x<1/2$,
$e^x\le 1+x+x^2$ we have
\begin{equation}\label{eqn:026}
\prod^{k}_{l=i}(1+\gamma_l)\le  e^{\sum_{l=i}^k\gamma_l} =e^{\sum_{l=0}^{k-i}
\kappa^l\gamma_i}\le  e^{\frac{\gamma_i}{1-\kappa}}\le
1+ \frac{\gamma_i}{1-\kappa}
+\left(\frac{\gamma_i}{1-\kappa}\right)^2.
\end{equation}
Hence
(\ref{eqn:014}),
(\ref{eqn:019}), (\ref{eqn:020}) and (\ref{eqn:021}) become
\begin{equation}\label{eqn:014N}
2\gamma_i-2\theta_i-2\kappa\theta_i-(1+2\kappa\theta_i)
\left( \frac{\gamma_i}{1-\kappa}
+\left(\frac{\gamma_i}{1-\kappa}\right)^2\right)>0,
\end{equation}
where we used the fact for $j\ge i+1$, $\theta_j\le \kappa\theta_i$.
\begin{equation}\label{eqn:019N}
\frac{\gamma_i}{2}+\eta_i- \frac{\gamma_i}{1-\kappa}
-\left(\frac{\gamma_i}{1-\kappa}\right)^2>0,
\end{equation}
\begin{equation}\label{eqn:020N}
\frac{\gamma_i}{2}+\frac{\theta_i}{2}+\frac{\eta_i}{2}-
 \frac{\gamma_i}{1-\kappa}
-\left(\frac{\gamma_i}{1-\kappa}\right)^2>0,
\end{equation}
\begin{equation}\label{eqn:021N}
\frac{\gamma_i}{2}+2{\theta_i}-\frac{\eta_i}{2}
- \frac{\gamma_i}{1-\kappa}
-\left(\frac{\gamma_i}{1-\kappa}\right)^2>0.
\end{equation}

Combining (\ref{eqn:024}) and (\ref{eqn:025}),(\ref{eqn:014N}),
(\ref{eqn:019N}), (\ref{eqn:020N}) and (\ref{eqn:021N}) become
\begin{equation}\label{eqn:014NN}
2\gamma_1-2\theta_1(1+\kappa)-(1+2\kappa\theta_1)
\left( \frac{\gamma_1}{1-\kappa}
+\left(\frac{\gamma_1}{1-\kappa}\right)^2\right)>0,
\end{equation}
\begin{equation}\label{eqn:019NN}
\frac{\gamma_1}{2}+\frac{3}{2}\theta_1- \frac{\gamma_1}{1-\kappa}
-\left(\frac{\gamma_1}{1-\kappa}\right)^2>0,
\end{equation}
\begin{equation}\label{eqn:020NN}
\frac{\gamma_1}{2}+\frac{5}{4}\theta_1-
 \frac{\gamma_1}{1-\kappa}
-\left(\frac{\gamma_1}{1-\kappa}\right)^2>0.
\end{equation}
Note that if (\ref{eqn:019NN}) is satisfied so is (\ref{eqn:020NN}). Thus we only have two conditions
left to satisfy, namely (\ref{eqn:014NN}) and (\ref{eqn:019NN}). At this point we can choose
\begin{equation}\label{eqn:027}
\gamma_1=\kappa^2(1-\kappa) \ \hbox{ and }\ \theta_1=\frac{\kappa^2(1-\kappa)}{2(1+4\kappa)},\
\hbox{ provided }\ 4\kappa^2(1-\kappa)<\alpha.
\end{equation}
Recalling that $\kappa<\frac{1}{16}$, a straightforward calculation shows that
\begin{eqnarray}\label{eqn:028}
2\gamma_1 & - & 2\theta_1(1+\kappa)-(1+2\kappa\theta_1)
\left( \frac{\gamma_1}{1-\kappa}
+\left(\frac{\gamma_1}{1-\kappa}\right)^2\right)\\[3mm]
& \ge & 2\kappa^2(1-\kappa)-2\theta_1(1+\kappa)-\kappa^2(1+2\kappa\theta_1)(1+\kappa^2) \nonumber\\
[3mm]
& \ge & \frac{\kappa^3}{1+4\kappa}(2-8\kappa-6\kappa^2)\ge \frac{\kappa^3}{1+4\kappa},\nonumber
\end{eqnarray}
and
\begin{eqnarray}\label{eqn:029}
\frac{\gamma_1}{2}+\frac{3}{2}\theta_1- \frac{\gamma_1}{1-\kappa}
-\left(\frac{\gamma_1}{1-\kappa}\right)^2 &\ge & \frac{1}{2}\left( \kappa^2(1-\kappa) +3\theta_1-
2\kappa^2-2\kappa^4\right)\\[3mm]
&\ge& \frac{\kappa^2}{4(1+4\kappa)}\left(1-13\kappa-12\kappa^2-16\kappa^3\right)\nonumber\\[3mm]
&\ge &\frac{\kappa^2}{32(1+4\kappa)}.\nonumber
\end{eqnarray}

Inequalities (\ref{eqn:028}), (\ref{eqn:028}) combined with (\ref{eqn:004A}), (\ref{eqn:8.68}),
(\ref{eqn:8.68A}), (\ref{eqn:011}),
(\ref{eqn:014}), (\ref{eqn:019}),
(\ref{eqn:020}) and (\ref{eqn:021}) show that for $\kappa$ such that
$4\kappa^2(1-\kappa)<\alpha$
\begin{equation}\label{eqn:030}
\beta(0,\frac{r_{k+1}}{4})\le C r^{\frac{\kappa^3}{1+4\kappa}}
\ \hbox{ where }\ r_{k+1}= r^{\prod^{k}_{l=1}(1+\gamma_l)}.
\end{equation}
Note that (\ref{eqn:027}) and (\ref{eqn:028}) ensure that
\begin{equation}\label{eqn:031}
\prod^{k}_{l=1}(1+\gamma_l)\le 1+\kappa+\kappa^2.
\end{equation}
Therefore for $t=\frac{r_{k+1}}{4}$ (\ref{eqn:030}) yields
\begin{equation}\label{eqn:032}
\beta(0,t)\le C t^{\frac{\kappa^3}{(1+4\kappa)(1+\kappa+\kappa^2)}},
\end{equation}
where $C$ is a constant that depends on $n$, $m$, $\alpha$ and
our specific choice of $\kappa$.

\section{On the flatness of asymptotically optimally doubling measures}

Recall the following result from Preiss \cite{Pr}. See also
\cite{DeL} [Propositions 6.18 and 6.19] for more details.

\begin{thm}\label{teopreiss}
There exists a constant $\ve_0>0$ depending only on $n$ and $d$ such
that if $\nu$ is an $n$-uniform measure on $\RR^m$ (normalized so
that $\nu(B(x,r))= r^n$ for all $x\in\supp(\nu)$, $r>0$) such that
its tangent measure $\lambda$ at $\infty$ satisfies
\begin{equation}\label{infty}
\min_{L\in G(n,m)} \int_{B(0,1)}\dist(x,L)^2d\lambda(x)  \leq \ve_0^2,
\end{equation}
then $\nu$ is flat.
Here $G(n,m)$ stands for the collection  of all
$n$-planes in $\RR^m$, $\lambda$ is normalized so that
$\lambda(B(x,r))= r^n$ for all $x\in\supp(\lambda)$, $r>0$.
\end{thm}

We need to define a smooth version of the usual coefficients
$\beta_2$. To this end, let $\vphi$ be a $C_c^\infty$ radial
function with $\chi_{B(0,2)}\leq \vphi \leq \chi_{B(0,3)}$.
Let $B=B(x_0,r)$ be a ball with centered at $x_0\in\supp(\mu)$. We denote by
\begin{equation}
\wt\beta_{2,\mu}(B) =
\min_{L\in G(n,m)} \biggl( \frac1{r^{n+2}}\int
\vphi\biggl(\frac{|x-x_0|}r\biggr)
\dist(x,L)^2d\mu(x)\biggr)^{1/2}.
\end{equation}

The following two theorems are the key tools in the proof of
Theorem \ref{thm:1.8A}. We postpone their proofs to the end of the section.
We first indicate how they are used to prove Theorem \ref{thm:1.8A}.

\begin{thm}\label{mth}
Let $\mu$ be an asymptotically optimally doubling measure supported on
$\Sigma\subset\RR^m$.
Let $K\subset\RR^m$ be compact and suppose that
\begin{equation}\label{eqn:04.1}
C_0^{-1}r^n\leq \mu(B(x,r))\leq C_0r^n\qquad\mbox{for $x\in
K\cap\Sigma$, $0<r\leq\diam(K)$}.
\end{equation}
For any $\eta>0$, there exists $\delta>0$ depending only on $\eta$, $n$, $m$,
$\mu$, $K$ and  $C_0$ such that if $B$ is a ball contained in $K$
and centered at $K\cap\Sigma$ with
$\wt\beta_{2,\mu}(B)\leq\delta$, then $\wt\beta_{2,\mu}(P)\leq\eta$
for any ball $P\subset B$ centered at $K\cap\Sigma$.
\end{thm}

\begin{thm}\label{b2t}
Let $\mu$ be an asymptotically optimally doubling measure supported on
$\Sigma\subset\RR^m$. Assume that $0\in \Sigma$
Let $K\subset\RR^m$ be a compact set such that $B(0,2)\subset K$, and
suppose that
\begin{equation}\label{eqn:04.2}
C_0^{-1}r^n\leq \mu(B(x,r))\leq C_0r^n\qquad\mbox{for $x\in
K\cap\Sigma$, $0<r\leq\diam(K)$}.
\end{equation}
Given $\epsilon>0$, there exists $\delta\in (0,\ve_0)$ depending only on
$\epsilon$, $n$, $m$, $\mu$, $K$ and  $C_0$ such that if
$\wt\beta_{2,\mu}(B)\leq\delta$,
for every ball $B\subset B(0,2)$ centered at $K\cap\Sigma$ then there exists
$R>0$ such that $\theta (x,r)<\epsilon$ for all $x\in \Sigma\cap B(0,1)$ and
$r<R$.
\end{thm}

\begin{cor}\label{cort}
Let $\mu$ be an asymptotically optimally doubling measure supported on
$\Sigma\subset\RR^m$
Let $K\subset\RR^m$ be compact set and suppose that
\begin{equation}\label{eqn:04.3}
C_0^{-1}r^n\leq \mu(B(x,r))\leq C_0r^n\qquad\mbox{for $x\in
K\cap\Sigma$, $0<r\leq\diam(K)$}.
\end{equation}
Given $\epsilon>0$, there exists $\delta\in (0,\ve_0)$ depending only on
$\epsilon$, $n$, $m$, $\mu$, $K$ and  $C_0$ such that if
$\wt\beta_{2,\mu}(B(x_0, 4R_0))\leq\delta$,
where $x_0\in \Sigma$  and $B(x_0, 4R_0)\subset K$, then there exists
$R>0$ such that $\theta (x,r)<\epsilon$ for all $x\in \Sigma\cap B(x_0,2R_0)$
and $r<R$, i.e.\ $\Sigma\cap B(x_0, 2R_0)$ is $\epsilon$-Reifenberg flat.
\end{cor}

\noindent{\bf Proof of Theorem \ref{thm:1.8A}}:
First note that (\ref{eqn:7.2AA}) ensures that condition
(\ref{eqn:04.3}) is satisfied. It also implies
that the density of $\mu$
exists and equals 1 everywhere. Therefore Preiss' work (see \cite{Pr})
yields that $\Sigma$ is $n$-rectifiable. Furthermore
$\mu={\cal H}^n\res\Sigma$. Thus given $\eta\in (0, \ve_0)$
for ${\cal H}^n$- a.e
$x\in \Sigma$ there
exists $\rho>0$ such that for $r<\rho$, $\theta(x,r)\le \eta$.
Let
\begin{equation}\label{defnr}
{\cal R}=\{x\in \Sigma:\limsup_{r\rightarrow 0}\theta(x,r)=0\}
\end{equation}
Note that ${\cal H}^n({\cal S})=0$ where ${\cal S}=\Sigma\backslash {\cal R}$.
For $x_0\in {\cal R}$ there exists
$R_0$ is such that $\theta(x_0, r)\le \eta$
for $r\le 8R_0$. This implies that
$\wt\beta_{2,\mu}(B(x_0, 4R_0))\leq C\eta$, where $C$ only depends on $C_0$.
For $\epsilon\in (0,\delta(n,m))$ where $\delta(n,m)$ is as in Theorem
\ref{thm:7.1}, by Corollary \ref{cort} we can find $\eta$ so that
$C\eta\le \delta\le \ve_0$, which ensures that $\Sigma\cap B(x_0, 2R_0)$ is
$\delta(n,m)$ Reifenberg flat. We use Theorem \ref{thm:7.1} to conclude
that $\Sigma\cap B(x_0, R_0)$ is a $C^{1,\beta}$ $n$-dimensional submanifold.
In particular this implies that ${\cal R}$ is open in $\Sigma$ because,
$\Sigma\cap B(x_0, 2R_0)\subset {\cal R}$.\qed

To prove Theorem \ref{mth} we need the following result:

\begin{lem} \label{mlem}
Let $\mu$ be an asymptotically optimally doubling measure on $\RR^m$.
Let $K\subset\RR^m$ be compact and let $\delta_0$ be any positive
constant. Suppose that
$$C_0^{-1}r^n\leq \mu(B(x,r))\leq C_0r^n\qquad\mbox{for $x\in
K\cap\Sigma$, $0<r\leq\diam(K)$}.$$
 There exists some constant
$\ve_1$ depending on $\ve_0$ and $C_0$ (but not on $\delta_0$) and
an integer $N>0$ depending only on $\mu$, $K$, $C_0$, and
$\delta_0$, such that if $B$ is a ball centered at $\Sigma$ such
that $2^NB\subset K$ and
\begin{equation}\label{eqn:ml}
\wt\beta_{2,\mu}(2^kB)\leq \ve_1\qquad \mbox{for $1\leq k\leq N$},\ \
\hbox{then}\ \  \wt\beta_{2,\mu}(B)\leq\delta_0.
\end{equation}
\end{lem}

\begin{proof}
Suppose that the integer $N$ does not exist. Then there exists a
sequence of points $\{x_j\}\subset K\cap\Sigma$ and balls $B_j:=
B(x_j,r_j)$ such that $2^jB_j\subset K$, and
$$\wt\beta_{2,\mu}(2^kB_j)\leq \ve_1\qquad \mbox{for
$1\leq k\leq j$},$$ but $\wt\beta_{2,\mu}(B_j)>\delta_0$. Clearly,
$r_j\to0$ as $j\to\infty$. For each $j\geq1$, consider the blow
up measure $\mu_j$ defined by
$$
\mu_j(A) = \frac{\mu(r_jA + x_j)}{\mu(B_j)}
$$
Extracting a subsequence if necessary, we may assume that
$\{\mu_j\}$ converges weakly to another measure $\nu$, which by
\cite{KT} [Theorem 2.2] is $n$-uniform. We claim that
\begin{equation}\label{claim1}
\wt\beta_{2,\nu}(B(0,2^k))\lesssim \ve_1 \qquad\mbox{for all
$k\geq0$}
\end{equation}
and
\begin{equation}\label{claim2}
\wt\beta_{2,\nu}(B(0,1))\gtrsim \delta_0.
\end{equation}

Assume the claim for the moment. It is easy to check that \rf{claim1}
implies that the tangent measure $\lambda$ of $\nu$ at $\infty$
satisfies
$$\min_{L\in G(n,m)} \int_{B(0,1)}\dist(x,L)^2d\lambda(x)  \leq \ve_0^2,$$
(assuming $\ve_1\leq\ve_0$ small enough) and so $\nu$ is flat by
Theorem \ref{teopreiss}. This contradicts \rf{claim2}, and the lemma
follows.

Let us prove \rf{claim1}. Let $B(0,2^k)$ be fixed. Extracting a
subsequence of $\{\mu_j\}$, we may assume that the $n$-planes $L_j$
which minimize $\wt\beta_{2,\mu_j}(B(0,2^k))$ converge in the
Hausdorff metric to another $n$-plane $L$, and then it easily
follows that
\begin{equation}\label{lim1}
\biggl| \int \vphi\biggl(\frac{|x|}{2^k}\biggr)
\dist(x,L_j)^2d\mu_j(x) - \int \vphi\biggl(\frac{|x|}{2^k}\biggr)
\dist(x,L)^2d\nu(x)\biggr| \to 0\qquad \mbox{as}\quad j\to\infty.
\end{equation}
Notice also that
\begin{align}
\frac1{2^{k(n+2)}} \int \vphi\biggl(\frac{|x|}{2^k}\biggr)
\dist(x,L_j)^2d\mu_j(x) & = \frac1{2^{k(n+2)}\mu(B_j)}\int
\vphi\biggl(\frac{|x-x_j|}{2^kr_j}\biggr)\,
\dist\Bigl(\frac{x-x_j}{r_j},L_j\Bigr)^2d\mu(x)\nonumber \\
& = \frac1{2^{k(n+2)}r_j^2\mu(B_j)}\int
\vphi\biggl(\frac{|x-x_j|}{2^kr_j}\biggr)\,
\dist(x,x_j+r_jL_j)^2d\mu(x) \nonumber\\
& \approx \frac1{(2^{k}r_j)^{n+2}}\int
\vphi\biggl(\frac{|x-x_j|}{2^kr_j}\biggr)
\dist(x,x_j+r_jL_j)^2d\mu(x) \nonumber\\
& \leq \ve_1^2,
\end{align}
since $x_j+r_jL_j$ is the $n$-plane that minimizes
$\wt\beta_{2,\mu}(B(x_j,2^kr_j))$. Inequality \rf{claim1} follows
from \rf{lim1} and the preceding estimate.

The proof of \rf{claim2} is analogous. Now let $L$ be
an arbitrary $n$-plane. Then we have
\begin{align}
\int \vphi(|x|)\,\dist(x,L)^2d\nu(x) & = \lim_{j\to\infty} \int
\vphi(|x|)\, \dist(x,L)^2d\mu_j(x) \nonumber\\
& = \lim_{j\to\infty}\frac1{\mu(B_j)} \int
\vphi\biggl(\frac{|x-x_j|}{r_j}\biggr)\,
\dist\Bigl(\frac{x-x_j}{r_j},L\Bigr)^2d\mu(x) \nonumber\\
& = \lim_{j\to\infty} \frac1{r_j^2\mu(B_j)} \int
\vphi\biggl(\frac{|x-x_j|}{r_j}\biggr)\, \dist(x,x_j+r_j L)^2d\mu(x)
\gtrsim \delta_0^2,
\end{align}
since $\wt\beta_{2,\mu}(B_j)>\delta_0$.
\end{proof}\qed\\

\noindent{\bf Proof of Theorem \ref{mth}:}
Let $\ve_1$ be the constant given by Lemma \ref{mlem}, and set
$\delta_0 = \min(\ve_1,\eta)$ (recall that $\ve_1$ is independent of
$\delta_0$). Let $N$ be the corresponding integer given by the same
lemma.

If $\delta$ is chosen small enough, then we clearly have
$\wt\beta_{2,\mu}(P)\leq\min(\ve_1,\eta)$ for any ball $P$ centered
at any point in $B\cap\Sigma$ with $r(P)\geq 2^{-N}r(B)$. By
the preceding lemma, by induction on $j\geq0$ we infer that
$\wt\beta_{2,\mu}(P)\leq\min(\ve_1,\eta)$ for any ball $P$ centered at
$B\cap\Sigma$ with
radius $r(P)$ such that $2^{-j-1}r(B) \leq r(P) \leq2^{-j}r(B)$
(where $r(B)$ stands for the radius of $B$). \qed\\

\noindent{\bf Proof of Theorem \ref{b2t}:} We argue by contradiction.
Suppose that there exists $\ve_1>0$ such that for each $i\ge i_0$
and each ball $B\subset B(0,2)$ centered in $K\cap \Sigma$,
$\wt\beta_{2,\mu}(B)\le 2^{-i}\le \ve_0$ but there are
$x_i\in \Sigma\cap B(0,1)$ and $r_i>0$ with $\lim_{i\rightarrow \infty}r_i=0$,
so that $\theta(x_i,r_i)\ge \ve_1$, i.e $\theta_{\Sigma_i}(0,1)\ge\ve_0$,
where $\Sigma_i=\frac{1}{r_i}(\Sigma -x_i)$.
Consider the blow up sequence $\{\mu_i\}$ defined by
\begin{equation}\label{blowup}
\mu_i(E)=\frac{\mu(r_iE+x_i)}{\mu(B(x_i,r_i))}
\end{equation}
Modulo passing to a subsequence Theorem 2.2 in \cite{KT} ensures that
$\mu_i$ converges weakly to a Radon measure $\mu_\infty$ which is
$n$-uniform. Moreover
$\Sigma_i$ converges in the Hausdorff distance
sense to $\Sigma_\infty= \supp\mu_\infty$ uniformly on compact subsets.
Therefore $\theta_{\Sigma_\infty}(0,1)\ge \epsilon_0/2$.
Statement (\ref{lim1}) guarantees that for $r>0$
$\tilde \beta_{2,\mu_i}(B(0,r))$ converges to
$\tilde\beta_{2,\mu_\infty}(B(0,r))$. Since  for $r>0$ there exists
$i_r$ so that
for $i\ge i_r$
$\tilde \beta_{2,\mu_i}(B(0,r))\le 2^{-i}$ then
$\tilde\beta_{2,\mu_\infty}(B(0,r))=0$ for every $r>0$.
Thus the support of $\mu_\infty$, $\Sigma_\infty$ is contained in an $n$-plane.
Since $\mu_\infty$ is $n$-uniform (and flat at infinity),
then $\Sigma_\infty$ is an $n$-plane, which contradicts the fact that
$\theta_{\Sigma_\infty}(0,1))\ge \epsilon_0/2$. \qed

\vskip 1in

\noindent\footnotesize{David Preiss, Mathematics Institut, University
of Warwick, Coventry CV4 7AL, UK.\\
 E-mail:d.preiss@warwick.ac.uk}\\

\noindent\footnotesize{Xavier Tolsa, ICREA and Departament de Matem\`atiques,
Universitat Aut\`onoma de Barcelona, 08193 Bellaterra . Barcelona,
Catalonia.\\
 E-mail: xtolsa@mat.uab.cat}\\

\noindent\footnotesize{Tatiana Toro,
Department of Mathematics,University of Washington,
Box 354350, Seattle, WA 98195-4350.\\
 E-mail: toro@math.washington.edu}

\end{document}